\documentclass[11pt]{article}
\usepackage{amscd,amssymb,stmaryrd}
\usepackage{amsmath}
\usepackage{graphicx}
\usepackage{subcaption}
\usepackage[utf8]{inputenc}
\usepackage[export]{adjustbox}
\usepackage{wrapfig}
\usepackage{amstext}
\usepackage{pstricks,pst-node,pst-plot,pst-coil}
\usepackage{amsthm}
\usepackage{amsmath}
\usepackage{color}
\usepackage{mathtools}
\usepackage{hyperref}
\usepackage[english]{babel}
\usepackage{booktabs}
\usepackage{mathrsfs}
\usepackage{comment}
\usepackage{float}

\newcommand\norm[1]{\left\|#1\right\|}
\newcommand\abs[1]{\lvert#1\rvert}

\newcommand{\tforall}{\text{ for all }}

\newcommand{\ms}{\text{CEM}}
\newcommand{\tand}{\quad \text{and} \quad}
\theoremstyle{definition}
\newtheorem{theorem}{Theorem}[section]
\newtheorem{proposition}[theorem]{Proposition}
\newtheorem{corollary}[theorem]{Corollary}
\newtheorem{lemma}[theorem]{Lemma}

\theoremstyle{remark}
\newtheorem*{remark}{Remark}

\DeclareMathOperator{\spa}{span}

\title{Computational multiscale methods for parabolic wave approximations in heterogeneous media}
\author{
Eric Chung\thanks{Department of Mathematics, The Chinese University of Hong Kong, Shatin, Hong Kong}\ , \quad Yalchin Efendiev\thanks{Department of Mathematics and Institute of Scientific Computing, Texas A\&M University, College Station, TX 77843, USA}\ , \quad Sai-Mang Pun\thanks{Department of Mathematics, Texas A\&M University, College Station, TX 77843, USA}\ , \quad and \quad Zecheng Zhang\thanks{Department of Mathematics, Purdue University, West Lafayette, IN 47906, USA}}
\begin{document}
\maketitle
%% -- ---------------------------------------------------------------------
%% -- ---------------------------------------------------------------------
\begin{abstract}
In this paper, we develop a computational multiscale to solve the parabolic wave approximation with heterogeneous and variable media. 
Parabolic wave approximation is a technique to approximate the full wave equation. One benefit of the method is that: one wave propagation direction can be taken as an evolution direction, and we then can discretize it using a classical scheme like Backward Euler. 
Consequently, we obtain a set of quasi-gas-dynamic (QGD) models with different heterogeneous permeability fields. 
Then, we employ constraint energy minimization generalized multiscale finite element method (CEM-GMsFEM) to perform spatial discretization for the problem. 
The resulting system can be solved by combining the central difference in time evolution. Due to the variable media, we apply the technique of proper orthogonal decomposition (POD) to further the dimension of the problem and solve the corresponding model problem in the POD space instead of in the whole multiscale space spanned by all possible multiscale basis functions. We prove the stability of the full discretization scheme and give the convergence analysis of the proposed approximation scheme. Numerical results verify the effectiveness of the proposed method.
\end{abstract}
%% -- ---------------------------------------------------------------------
%% -- ---------------------------------------------------------------------

%\tableofcontents

\section{Introduction}
Parabolic wave approximation have been used to approximate wave equations with a preferred direction \cite{bamberger1988higher,bamberger1988parabolic}.  
To be more specific, we study the approximation of the full wave equation:
\begin{eqnarray}
\rho \partial_{tt} u - \nabla \cdot (\mu \nabla u) = 0.
\label{eqn:wave}
\end{eqnarray}
Here, $\nabla$ denotes the gradient operator in $\mathbb{R}^3$; $\rho$ and $\mu$ are positive functions in $\mathbb{R}^3$. We then consider the parabolic approximation of \eqref{eqn:wave} in $D \times \Omega \subseteq \mathbb{R}^3$: for $(z,x) \in D \times \Omega$ with the boundary $\Gamma := \partial (D \times \Omega)$, find $v = \sigma^{1/2} u$ such that 
\begin{eqnarray}
\begin{split}
c^{-1} \partial_{tt} v  + \partial_t (\partial_{z} v) - \frac{1}{2} \nabla_x \cdot \left (c \nabla_x v \right ) & = 0 & \quad \text{in } D \times \Omega \times (0,T], \\
v(z,x,0) & = v_0(z,x) & \quad \text{in } D \times \Omega, \\
\partial_t v(z,x,0) & = v_1(z,x) & \quad \text{in } D \times \Omega, \\
v(z,x,t) & = g(z,x,t) & \quad \text{in } \Gamma \times (0,T].
\end{split}
\label{eqn:model}
\end{eqnarray}
Here, $\nabla_x$ denotes the gradient operator defined in the bounded domain $\Omega \subset \mathbb{R}^2$; $D = [d_1,d_2] \subset \mathbb{R}$ is a bounded domain; $v_0$ and $v_1$ are initial conditions; $g$ is boundary condition and $T>0$ is a given terminal time.

Wave equations of this type have a propagation direction $z$ which plays a role of time \cite{bamberger1984paraxial}.
Many real-world problems can be solved by the parabolic approximation; in particular in the areas of geology \cite{claerbout1976fundamentals}, under water acoustics \cite{blomgren2002super,brock1977modifying, mcdaniel1975parabolic,mcdaniel1975propagation} and optics  \cite{cole1975modern,hasegawa1973transmission,hudson1980parabolic}. There are two benefits for using the parabolic approximation instead of the full wave equation \cite{bamberger1988higher,bamberger1988parabolic}: (i) it is easier to be realized and computationally more efficient; and (ii) it makes it possible to use the approximation as an evolution equation in $z$ direction. 
The second property makes it possible to solve the problem as an equation evolving in $z$ direction; we hence can discretize ($\partial_z v$) by applying classical difference scheme. Consequently, we obtain a quasi-gas-dynamic (QGD) equation which have been thoroughly studied in literature \cite{chetverushkin2020computational,chetverushkin2019compact, chetverushkin2018kinetic}.

It is relatively common in the geological problems \cite{claerbout1976fundamentals} that the medium under consideration is highly variable and heterogeneous; that is, the function field $c$ is non-homogeneous given a cross section in $z$ and is fast-changing in $z$ direction. 
This brings in two difficulties. The first to mention is the multiscale property brought by the heterogeneous field. 
The second difficulty is the intense work in solving a QGD model given a $z$ cross-section in the $z$ evolution.
Throughout the work, we study \eqref{eqn:model} with variable and heterogeneous media.
We will derive scheme in discretizing the time and $z$ evolution and also provide our approach to solve the problems mentioned before. 

As we have discussed, the model has heterogeneous in $z$ direction. Directly solving this problem on fine mesh can capture the multiscale features; however, this is computationally intense and this issue becomes exacerbate when people are solving time-dependent problems. Therefore, many methods which solve the multiscale problems on coarser mesh have been proposed.
These include
 homogenization-based approaches \cite{cances2015embedded, chen2019homogenize,chen2019homogenization,fu2019edge,le2014msfem,le2014multiscale,salama2017flow}, 
multiscale
finite element methods \cite{hkj12,hw97,jennylt03,jennylt05}, 
generalized multiscale finite element methods (GMsFEM) \cite{chung2016adaptiveJCP,MixedGMsFEM,WaveGMsFEM,chung2018fast,chung2015goal,GMsFEM13,gao2015generalized}, constraint energy minimizing GMsFEM (CEM-GMsFEM) \cite{chung2018constraint, chung2018constraintmixed}, nonlocal
multi-continua (NLMC) approaches \cite{NLMC},
metric-based upscaling \cite{oz06_1}, heterogeneous multiscale method \cite{abe07,ee03}, localized orthogonal decomposition (LOD) \cite{henning2012localized,maalqvist2014localization}, equation free approaches \cite{rk07,skr06,srk05}, computational continua \cite{fafalis2018computational,fish2010computational,fish2005multiscale}, hierarchical multiscale method \cite{brown2013efficient,hs05,tan2019high}, 
and so on. Some of these approaches, such as homogenization-based
approaches, are designed for problems with scale 
separation. In this work, we apply the CEM-GMsFEM \cite{chung2018constraint, chung2018constraintmixed} and provide the convergence analysis of our proposed scheme based on the coarse mesh convergence results of the CEM-GMsFEM.

The second difficulty of the problem is the variable media. If we discretize $z$ evolution using some classical difference scheme, each $z$ level is a QGD model with heterogeneous. This model can be solved in the framework of CEM-GMsFEM \cite{chetverushkin2020computational}; however the coarse scale basis evaluation is time consuming; in particular, this process will be repeated for each level of $z$. We hence proposed the proper orthogonal decomposition (POD) technique \cite{kunisch2001galerkin} which is target to find low dimensional subspace such that the error of the orthogonal projection is minimized in the sense of the norm induced by the inner product of the original space. To be more specific, the proposed method can be summarized as follows:

We proposed a method to solve a parabolic-wave model with variable and heterogeneous media. We first apply the backward Euler scheme to discretization the term $\partial_z v$. 
This leads to a set of two dimensional (space) QGD models \eqref{eqn:quasi-dis}; we call this level of discretization the quasi-time scheme and an unconditional stability result is established. 
To further discretize the problem, we apply the central scheme to deal with the time derivative $v_{tt}$, $v_t$; 
and then use CEM-GMsFEM in space on coarse scale to capture the heterogeneous brought by the media. 
We then prove that the full discretization scheme is stable in an energy norm under some CFL condition by using an inverse inequality in the multiscale space.

The key of the CEM-GMsFEM method is to construct the CEM basis.
The standard procedure is first to build the auxiliary multiscale basis by solving local spectral problems in coarse mesh; we then can construct the CEM basis by evaluating a set of energy minimization problems. Due to the variable media, we need to construct a set of basis for each QGD model and the solution of the full discretized scheme is in the space of all multiscale basis. This is time consuming; and we hence can apply the POD technique to find the best set of orthogonal basis in the sense of $L_2$ minimization; that is, the projection error of the full discretized solution onto the POD basis is optimal in the norm induced by the original space. In practice, we collect CEM basis for some QGD models and then construct POD basis of the space spanned by all CEM basis; the POD models will finally be solved by using POD basis. A convergence analysis of the POD approximation is established and the numerical results prove the algorithm is successful.

The remainder of the paper is organized as follows. In Section \ref{model_problem}, we present some preliminaries of the model problem and briefly overview the framework of proper orthogonal decomposition. 
Section \ref{sec:gmsfem} is devoted to the multiscale methods and we will briefly overview the construction of multiscale basis function within the framework of the CEM-GMsFEM. In Section \ref{sec:analysis}, we present a complete analysis of the proposed computational multiscale method. We then present
some numerical results to demonstrate the efficiency of the proposed method in Section \ref{sec:numerics}. Concluding remarks are drawn in Section \ref{sec:conclusion}.

\section{Preliminaries}\label{model_problem}
In this section, we present some preliminaries of the model problem. For simplicity, we assume the homogeneous boundary condition $g = 0$ is equipped in the model problem \eqref{eqn:model}. The extension of inhomogeneous case is straightforward. The system \eqref{eqn:model} is the one we shall consider throughout the remainder of this paper and we shall develop computational multiscale method for efficiently simulating the problem. 
We remark that under appropriate regularity assumptions on initial and boundary conditions, the problem \eqref{eqn:model} has a unique solution such that 
$$ t \to v(z,x,t) \in W^{1,\infty} (0,T; L^2(D \times \Omega)) \cap L^\infty (0,T; L^2(D \times \Omega)),$$
$$z \to v(z,x,t) \in L^\infty (D, H^1(0,T; L^2(\Omega))).$$
The last property enables us to consider $z$ as an evolution direction. 
The result is an application of the semigroup theory and the Hille-Yoshida theorem. See \cite[Section 4]{bamberger1988parabolic} for more details. 
%We denote the gradient operator $\nabla v = \left ( \partial_{x_1} v ~ \partial_{x_3} v \right )^T$. 

Instead of the PDE formulation \eqref{eqn:model}, we consider its corresponding variational formulation. 
In the following, we treat $z$ as an evolution direction and use backward Euler method to discretize the term $\partial_{z} v$. We divide the domain $D = [d_1, d_2]$ along $z$-direction into $K$ pieces. We write, for $k = 0,1,\cdots, K$, 
$$v_k = v(z_k) \quad \text{with} \quad z_{k} = d_1 + k\Delta z \tand d_2 = d_1 + K\Delta z.$$ 
The quasi-time discretization of \eqref{eqn:model} reads: find $\theta_k \in V$ for $k = 1, \cdots, K-1$ such that 
\begin{eqnarray}
 \left ( \ddot{\theta}_k, w; z_k \right )_c + \frac{1}{\Delta z} \left ( \dot{\theta}_k - \dot{\theta}_{k-1}, w \right ) + \frac{1}{2} a(\theta_k,w; z_k) = 0 \quad \tforall w \in V:= H_0^1(\Omega)
\label{eqn:quasi-dis}
\end{eqnarray}
or equivalently 
$$\Delta z  \left ( \dot{\theta}_k, w; z_k \right )_c + \left ( \dot{\theta}_k , w \right ) + \frac{\Delta z }{2} a(\theta_k,w; z_k) = \left ( \dot{\theta}_{k-1}, w \right ) \quad \tforall w \in V, 
$$
where $(v,w; z_k)_c := \int_\Omega c^{-1}(z_k) v w ~ dx$, $(v,w) := \int_\Omega vw ~ dx$, and $a(v,w; z_k) := \int_\Omega c(z_k) \nabla v \cdot \nabla w~ dx$. Here, we denote $\dot{v} := \partial_t v$ and $\ddot{v} := \partial_{tt} v$. 
We assume that $c \in L^\infty(D \times \Omega)$ and we denote $c_{\max} := \norm{c}_{L^\infty(D \times \Omega)}$. 
Employing Galerkin's method and the method of energy estimate, one can show the well-posedness of the variational formulation \eqref{eqn:quasi-dis}. %See \cite[Chapter 7.2]{evans2010partial} for more details. 
%Assume $c \in L^\infty(\Omega)$ and there exist two constant $\beta$ and $\gamma$ such that $0 < \gamma \leq c(x) \leq \beta$ for almost all $x \in \Omega$. 
Denote $\norm{v} := \sqrt{(v,v)}$, $\norm{v}_{c(z_k)} := \sqrt{(v,v; z_k)_c}$, and $\norm{v}_{a(z_k)}:= \sqrt{a(v,v; z_k)}$. 
%$$(u,v)_c := \int_\Omega c^{-1} uv ~ dx \quad \text{and} \quad \norm{u}_c := \sqrt{(u,u)_c}.$$
We establish the following stability estimate for the quasi-time discretization \eqref{eqn:quasi-dis}. 
\begin{lemma}
Let $\{v_k\}_{k=0}^K \subseteq V$ solve the equation \eqref{eqn:quasi-dis}. Then, the following stability estimate holds: 
$$\norm{\dot{v_K}}_{L^2(0,T; L^2(\Omega))}^2  +  \Delta z \sum_{k=1}^K \mathcal{E}_k (v_k(T)) \lesssim \norm{\dot{v_0}}_{L^2(0,T; L^2(\Omega))}^2  + \Delta z \sum_{k=1}^K  \mathcal{E}_k(v_k(0)),$$
where we denote $\mathcal{E}_k(v) := \norm{\dot{v}}_{c(z_k)}^2 + \frac{1}{2} \norm{v}_{a(z_k)}^2$.
\label{semi_stability}
\end{lemma}
\begin{remark}
The term $\mathcal{E}_k(v_k(0)) = \norm{v_1(z_k)}_{c(z_k)}^2 + \frac{1}{2} \norm{v_0(z_k)}_{a(z_k)}^2$ on the right-hand side of the inequality above is determined by the initial conditions. 
\end{remark}
\begin{proof}[Proof of Lemma \ref{semi_stability}]
Taking $w =  \dot{v_k}$, we have 
\begin{eqnarray*}
\begin{split}
& \quad \frac{\Delta z}{2} \frac{d}{dt} \left ( \norm{\dot{v}_k}_{c(z_k)}^2 + \frac{1}{2} \norm{v_k}_{a(z_k)}^2 \right ) + \norm{\dot{v}_k}^2 -(\dot{v}_{k-1}, \dot{v}_{k}) \\
& = \Delta z \left ( \ddot{v}_k, \dot{v}_k ;z_k \right )_{c}+ (\dot{v}_k - \dot{v}_{k-1}, \dot{v_k}) + \frac{\Delta z}{2} a(v_k, \dot{v_k};z_k)  = 0. 
\end{split}
\end{eqnarray*}
Thus, we have 
$$ \quad \frac{\Delta z }{2} \frac{d}{dt} \left (  \norm{\dot{v}_k}_{c(z_k)}^2 + \frac{1}{2} \norm{v_k}_{a(z_k)}^2 \right )  + \norm{\dot{v}_k}^2  =   (\dot{v}_k, \dot{v}_{k-1}) \leq \norm{\dot{v}_k} \norm{\dot{v}_{k-1}}.$$
Multiplying by $2$ and integrating over $(0,T]$, we have 
\begin{eqnarray*}
\begin{split}
%& \quad \frac{\Delta z }{2} \frac{d}{dt} \left (  c_{\max}^{-1}\norm{\dot{v_k}}^2 + \frac{1}{2} \norm{v_k}_a^2 \right )  + \norm{\dot{v_k}}^2  =   (\partial_t v_k, \partial_t v_{k-1}) \leq \norm{\dot{v_k}} \norm{\dot{v_{k-1}}}, \\
& {\Delta z} \left ( \norm{\dot{v}_k(T)}_{c(z_k)}^2 + \frac{1}{2} \norm{v_k(T)}_{a(z_k)}^2 \right ) + 2 \int_0^T \norm{\dot{v}_k}^2 dt  \\
& \leq 2 \int_0^T \norm{\dot{v}_k} \norm{\dot{v}_{k-1}} dt + \Delta z \left ( \norm{\dot{v}_k (0)}_{c(z_k)}^2 + \frac{1}{2} \norm{v_k(0)}_{a(z_k)}^2 \right )\\
&\leq   \int_0^T \norm{\dot{v}_k}^2 dt + \int_0^T \norm{\dot{v}_{k-1}}^2 dt + \Delta z \left ( \norm{\dot{v}_k(0)}_{c(z_k)}^2 + \frac{1}{2} \norm{v_k(0)}_{a(z_k)}^2 \right ).\\
\end{split}
\end{eqnarray*}
Therefore, we have 
\begin{eqnarray*}
\begin{split}
\norm{\dot{v}_k}_{L^2(0,T; L^2(\Omega))}^2  & + {\Delta z} \left (  \norm{\dot{v}_k(T)}_{c(z_k)}^2 + \frac{1}{2} \norm{v_k(T)}_{a(z_k)}^2 \right ) \\
& \leq \norm{\dot{v}_{k-1}}_{L^2(0,T; L^2(\Omega))}^2  + \Delta z \left (\norm{\dot{v}_k(0)}_{c(z_k)}^2 + \frac{1}{2} \norm{v_k(0)}_{a(z_k)}^2 \right ).
\end{split}
\end{eqnarray*}
Summing over $k = 1, \cdots, K$, we obtain the desired result and this completes the proof. 
\end{proof}

\subsection{The Proper Orthogonal Decomposition}
\label{proper_orthogonal_decomposition}
In this section, we briefly introduce the proper orthogonal decomposition (POD) method. This method aims to generate optimally ordered orthogonal basis functions in the least squares sense for a given set of theoretical, experimental, or computational data. Reduced-order models or surrogate models are then obtained by truncating this set of optimal basis functions, providing considerable computational savings over the original high-dimensional problems. 

Let $X$ be a real Hilbert space endowed with inner product $(\cdot,\cdot)_X$ and norm $\norm{\cdot}_X$. We set $\mathcal{V} := \spa \{ y_1, y_2, \cdots, y_n \}$ with each $y_i \in X$ for $i \in \{ 1, \cdots, n\}$. We refer to $\mathcal{V}$ as ensemble consisting of the snapshots $\{ y_i \}_{i=1}^n$, at least one of which is assumed to be non-zero. Let $\{ \psi_k \}_{k=1}^{\mathcal{N}}$ denote a set of orthonormal basis functions of  $\mathcal{V}$ with $\mathcal{N} := \text{dim}(\mathcal{V}) \leq n$. Then, each member of the ensemble can be expressed as 
$$ y_j = \sum_{k=1}^{\mathcal{N}} (y_j, \psi_k)_X \psi_k$$
for each $j \in \{ 1, \cdots, n\}$. The POD method consists in choosing the orthonormal basis functions such that for every $\ell \in \{ 1 , \cdots, \mathcal{N} \}$ the mean square error between the elements $y_j$ (for any $j  \in \{ 1, \cdots, n\}$), and the corresponding $\ell$-th partial sum is minimized on average:
\begin{eqnarray}
\begin{split}
& \min_{ \{ \psi_k \}_{k=1}^{\ell} } \frac{1}{n} \sum_{j=1}^n \norm{y_j - \sum_{k=1}^{\ell} (y_j, \psi_k)_X \psi_k}_X^2 \\
& \text{subject to} \quad (\psi_k, \psi_t) = \delta_{kt} \quad \text{for any }k, t \in \{ 1,2,\cdots, \ell\}.
\end{split}
\label{eqn:pod-exp}
\end{eqnarray}
Here, $\delta_{kt}$ denotes the Kronecker-delta function. A solution $\{ \psi_k \}_{k=1}^{\ell}$ to \eqref{eqn:pod-exp} is called a POD-basis of rank $\ell$. We introduce the correlation matrix 
$$ K = \left ( \frac{1}{n} (y_j, y_i)_X \right ) \in \mathbb{R}^{n \times n}$$
corresponding to the snapshots $\{ y_j \}_{j=1}^n$. The matrix $K$ is positive semi-definite and has rank $\mathcal{N}$. The minimization problem \eqref{eqn:pod-exp} can be reduced to an eigenvalue problem 
\begin{eqnarray}
Kv = \lambda v.
\label{eqn:pod-cor}
\end{eqnarray}
We sort all the positive eigenvalues in a decreasing order as $\lambda_1 \geq \lambda_2 \geq \cdots \geq \lambda_{\mathcal{N}} >0$ and the associated eigenvectors are denoted by $v_k$ with $k = 1,\cdots, \mathcal{N}$. It can be shown that the POD-basis of rank $\ell \in \mathbb{N}^+$ with $\ell \leq \mathcal{N}$ is formed by 
\begin{eqnarray}
\varphi_k = \frac{1}{\sqrt{\lambda_k}} \sum_{j=1}^{n} (v_k)_j y_j \quad \text{for } k = 1,\cdots, \ell.
\label{eqn:pod-basis}
\end{eqnarray}
Here, $(v_k)_j$ is the $j$-th component of the eigenvector $v_k$. The basis functions $\{ \varphi_k \}_{k=1}^{\ell}$ form a POD-basis of rank $\ell$ and we have the following error formula. 

\begin{proposition}
Let $\lambda_1 \geq \lambda_2 \geq \cdots \geq \lambda_{\mathcal{N}} >0$ be the positive eigenvalues of $K$ in \eqref{eqn:pod-cor} and $v_1, \cdots, v_{\mathcal{N}} \in \mathbb{R}^n$ be the associated eigenvectors. Then, $\{ \varphi_k \}_{k=1}^{\ell}$ given by \eqref{eqn:pod-basis} forms a set of POD-basis of rank $\ell$ with $\ell \leq \mathcal{N}$. Moreover, we have the error formula 
$$ \frac{1}{n} \sum_{j=1}^n \norm{y_j - \sum_{k=1}^{\ell} (y_j , \varphi_k)_X \varphi_k}_X^2 = \sum_{k=\ell+1}^\mathcal{N} \lambda_k.$$
\end{proposition}

In practice, we shall make use of the decay property of eigenvalues in $\lambda_k$ and choose the first $\ell$ dominant eigenvalues such that the ratio $\zeta := \frac{\sum_{k=\ell+1}^{\mathcal{N}} \lambda_k}{\sum_{k=1}^{\mathcal{N}} \lambda_k}$ is small enough to achieve an expected accuracy, for instance $\zeta = 1\%$. One would prefer the eigenvalues decays as fast as possible so that one can ensure high accuracy with few POD basis functions. 

\section{Multiscale Method}\label{sec:gmsfem}
In this section, we develop the computational multiscale method in order to solve the parabolic wave approximation. For spatial discretization, we will apply the CEM-GMsFEM. In particular, for each node $z_i \in D$ along the $z$-direction, we will construct a set of multiscale basis functions in the spirit of CEM-GMsFEM. To further reduce the dimension of the multiscale space, we will perform POD procedure related to these CEM basis functions. 
%We emphasize that the multiscale basis functions and the corresponding space are defined with respect to the coarse grid of the domain. The multiscale method consists of two steps. First, we perform a spectral decomposition and form an auxiliary space. 
%Next, we construct a multiscale space for approximating the solution based on the auxiliary space.  We remark that these basis functions are locally supported in some coarse patches formed by some coarse elements. 
Once the multiscale space is constructed, one can use leapfrog scheme to discretize time derivatives and solve the resulting fully-discretized problem. 

\subsection{Spatial Discretization: CEM-GMsFEM}
In this section, we briefly outline the framework of CEM-GMsFEM and present the construction of the multiscale space. 
First, we introduce fine and coarse grids for the computational domain.
Let $\mathcal{T}^H = \{ K_i \}_{i=1}^N$ be a conforming partition of the domain $\Omega$ with mesh size $H>0$ defined by
$$H := \max_{K \in \mathcal{T}^H} \Big(\max_{x, y \in K} \abs{x-y}\Big).$$ 
We refer to this partition as coarse grid. We denote $N \in \mathbb{N}^+$ the total number of coarse elements. Subordinate to the coarse grid, we define the fine grid partition $\mathcal{T}^h$ (with mesh size $h \ll H$) by refining each coarse element $K \in \mathcal{T}^H$ into a connected union of finer elements. We assume the refinement above is performed such that $\mathcal{T}^h$ is also a conforming partition of the domain $\Omega$. Denote $N_c$ the number of interior coarse grid nodes of $\mathcal{T}^H$ and $\{ x_i \}_{i=1}^{N_c}$ the collection of interior coarse nodes in the coarse grid. 

\subsubsection{Spectral Decomposition}

We present the construction of the auxiliary multiscale basis functions. Let $K_i \in \mathcal{T}^H$ be a coarse block. Define $V(K_i)$ as the restriction of the abstract space $V$ on the coarse element $K_i$. For each $z_k \in D$, we consider a local spectral problem: Find $\sigma_j^{(i,k)} \in \mathbb{R}$ and $\phi_j^{(i,k)} \in V(K_i)$ such that 
\begin{eqnarray} \label{eqn:spectral}
a_i(\phi_j^{(i,k)}, v; z_k) = \sigma_j^{(i,k)} s_i(\phi_j^{(i,k)},v; z_k) \quad \text{for all}~ v \in V(K_i).
\end{eqnarray}
Here, $a_i: V(K_i) \times V(K_i)$ is a symmetric non-negative definite bilinear form and $s_i: V(K_i) \times V(K_i)$ is a symmetric positive definite bilinear form. 
We remark that the above problem is solved on the fine mesh in the actual computations. Based on the analysis, we choose 
$$ a_i(v,w; z_k) := \int_{K_i} c(z_k) \nabla v \cdot \nabla w ~ dx, \quad s_i(v,w; z_k) := \int_{K_i} \tilde \kappa(z_k) v w ~ dx ,$$
where $\tilde \kappa := \sum_{j=1}^{N_c} c(z_k) \abs{\nabla \chi_{j,k}^{\text{ms}}}^2$. 
The functions $\{ \chi_{j,k}^{\text{ms}} \}_{j=1}^{N_c}$ are the standard multiscale finite element basis functions which satisfy the partition of unity property. More precisely, $\chi_{j,k}^{\text{ms}}$ is the solution of the following system: 
\begin{eqnarray*}
\nabla \cdot (c(z_k) \nabla \chi_{j,k}^{\text{ms}}) & = 0 \quad &\text{in each } K \subset \omega_j, \\
\chi_{j,k}^{\text{ms}} & = g_j \quad &\text{on } \partial K \setminus \partial \omega_j,\\
\chi_{j,k}^{\text{ms}} & = 0 \quad &\text{on } \partial \omega_j. 
\end{eqnarray*}
The function $g_j$ is continuous and linear along the boundary of the coarse element. We assume that the eigenvalues $\sigma_j^{(i,k)}$ are arranged in ascending order and we pick $\ell_{i,k} \in \mathbb{N}^+$ corresponding eigenfunctions to construct the local auxiliary space $V_{\text{aux}}^{(i,k)}:= \text{span} \{ \phi_j^{(i,k)} : j = 1, \cdots, \ell_i \}$. 
We assume the normalization $s_i\left ( \phi_j^{(i,k)}, \phi_j^{(i,k)}; z_k \right ) = 1$. 
After that, we define the global auxiliary multiscale space $V_{\text{aux}}^k := \bigoplus_{i=1}^{N} V_{\text{aux}}^{(i,k)}$. We remark that the global auxiliary multiscale space is used to construct multiscale basis functions that are orthogonal to the auxiliary space with respect to the weighted $L^2$ inner product $s(\cdot,\cdot; z_k)$. 

Note that the bilinear form $s_i(\cdot, \cdot; z_k)$ defines an inner product with norm $\norm{\cdot}_{s(K_i; k)} := \sqrt{s_i(\cdot,\cdot; z_k)}$ in the local auxiliary space $V_{\text{aux}}^{(i,k)}$. Based on these local inner products and norms, one can naturally define a new inner product and norm for the global auxiliary space $V_{\text{aux}}$ as follows: for all $v, w \in V_{\text{aux}}^k$, 
\begin{eqnarray}
    s(v,w; z_k) := \sum_{i=1}^N s_i(v,w; z_k) \quad \text{and} \quad \norm{v}_{s(z_k)} := \sqrt{s(v,v; z_k)}.
    \label{l2norm}
\end{eqnarray}
Note that if $\{ \chi_{j,k}^{\text{ms}} \}_{j=1}^{N_c}$ is a set of bilinear partition of unity, then 
$\norm{v}_{s(z_k)} \leq H^{-1} (\max\{c\})^{1/2} \norm{v}$
for any $v \in L^2(\Omega)$. 
In addition, we define $\pi_k: L^2(\Omega) \to V_{\text{aux}}^k$ as the projection with respect to the inner product $s(\cdot,\cdot;z_k)$ such that 
$$ \pi_k(u) := \sum_{i=1}^N \sum_{j=1}^{\ell_i} s_i (u, \phi_j^{(i,k)};z_k) \phi_j^{(i,k)} \quad \text{for all } u \in L^2(\Omega). $$

\subsubsection{The construction of multiscale basis functions}
In this section, we present the construction of the multiscale basis functions. First, we define an oversampling region for each coarse element. Specifically, given a non-negative integer $m \in \mathbb{N}$ and a coarse element $K_i$, we define the oversampling region $K_{i,m} \subset \Omega$ such that 
$$ K_{i,m} := \left \{ \begin{array}{lr} 
K_i & \text{if } m = 0, \\
\displaystyle{\bigcup \{ K: K_{i,m-1} \cap K \neq \emptyset \}} & \text{if } m \geq 1.
\end{array} \right .$$

See Figure \ref{fig:mesh} for an illustration of oversampling region. For simplicity, we denote $K_i^+$ the oversampled region $K_{i,m}$ for some nonnegative integer $m$. 

\begin{figure}[ht]
\centering
\includegraphics[width = 2.5in]{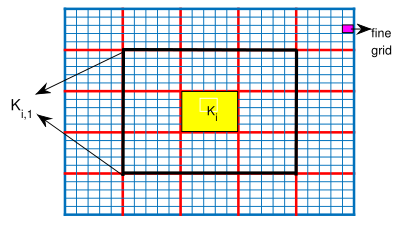}
\caption{Oversampling region with $m = 1$.} 
\label{fig:mesh}
\end{figure}

Recall that $V(K_i^+)$ is the restriction of $V$ on the coarse patch $K_i^+$. Let $V_0(K_i^+)$ be the subspace of $V(K_i^+)$ with zero trace on the boundary $\partial K_i^+$. 
For each eigenfunction $\phi_j^{(i,k)} \in V_{\text{aux}}^k$, we define the multiscale basis $\psi_{j,\text{ms}}^{(i,k)} \in V_0(K_i^+)$ to be the solution of the equation: 
% teleport: variational form of ms basis
\begin{eqnarray}
a(\psi_{j,\text{ms}}^{(i,k)}, v) + s\left(\pi( \psi_{j,\text{ms}}^{(i,k)}), \pi_k (v) \right )  =  s(\phi_j^{(i,k)},v) \quad \text{for all } v \in V_{0}(K_i^+) \label{eqn:msv}.
\end{eqnarray}
Then, the multiscale space is defined as 
$ V_{\text{ms}}^k := \text{span} \left \{ \psi_{j,\text{ms}}^{(i,k)}: i = 1,\cdots, N, ~ j = 1,\cdots, \ell_{i,k} \right \}$. 
By construction, we have $\text{dim}(V_{\text{ms}}^k) = \text{dim}(V_{\text{aux}}^k)$. After that, we define 
$$V_{\text{CEM}} := \text{span} \left \{  \psi_{j,\text{ms}}^{(i,k)}: i = 1,\cdots, N, ~ j = 1,\cdots, \ell_{i,k}, ~ k = 0, \cdots, K \right \}.$$
We will perform the procedure of POD for this space to construct the multiscale space that will be used in actual simulation.

\subsection{Construction of multiscale reduced basis functions using POD} 
In this section, we present the construction of multiscale reduced basis functions using the POD technique. First, we define $ \{ v_{\text{CEM},k} \}_{k=0}^K \subset V_{\text{CEM}}$ to be the solution of the following equation: 
\begin{eqnarray}
 \left ( \ddot{v}_{\text{CEM},k}, w; z_k \right )_c + \frac{1}{\Delta z} \left ( \dot{v}_{\text{CEM},k} - \dot{v}_{\text{CEM},k-1}, w \right ) + \frac{1}{2} a(v_{\ms,k},w; z_k) = 0 \quad \tforall w \in V_{\text{CEM}}.
\label{eqn:quasi-dis-cem}
\end{eqnarray}
\textcolor{red}{
In the abstract framework of POD, we set the space $X = L^2(0,T;H^1(\Omega))$.  
We define the inner product associated with $X$ as $\left \langle u, v \right \rangle : = \int_0^T (\nabla u, \nabla v)+(u, v) ~ dt$ for $u, v\in X$. The corresponding norm will then be defined as $\norm{\cdot}_{L^2(0,T;H^1(\Omega))} := \sqrt{\left \langle \cdot, \cdot \right \rangle}$.}
Next, we define the snapshot space $\mathcal{V}$ \textcolor{red}{as the collection of}: 
$$
y_j := v_{\text{CEM}} (t, z_{j-1}), \quad  y_{j+K+1} := \ddot{v}_{\text{CEM}}(t, z_{j-1}), \quad j = 1,\cdots, K+1, 
$$
and
$$
y_{j+2K+2} = \Tilde{\partial} \dot{v}_{\ms}(t, z_j) := 
\frac{\dot{v}_{\text{CEM}}(t, z_j)-\dot{v}_{\text{CEM}}(t, z_{j-1})}{\Delta z}, \quad j = 1,\cdots ,K.$$ 
\textcolor{red}{Then, we perform the POD procedure  on the snapshot space $\mathcal{V}$ as described in Section \ref{proper_orthogonal_decomposition}. It should be noted that the correlation matrix is defined as $K_{ij} := \frac{1}{3K+2}(y_j, y_i)_{H^1}$.}
We denote $\{\psi_k\}_{k = 1}^\mathcal{N}$ the corresponding POD basis functions with cardinality $\mathcal{N} \in \mathbb{N}^+$. 
We then define $V_{\text{POD}}^\ell := \spa\{\psi_k\}_{k = 1}^\ell$ for a given positive number $\ell \leq \mathcal{N}$. 

\subsection{Fully Discretization}
In this section, we present the fully discretization for the problem \eqref{eqn:quasi-dis}. 
We can further consider the discretization in time. We divide $(0,T]$ into $N$ pieces and write 
$$ v^n = v(t^n) \quad \text{with} \quad t^n = n\Delta t \quad \text{and} \quad \Delta t = \frac{T}{N}.$$
Specifically, we use first- and second-order central difference schemes. 
We define $\tau := \Delta t \Delta z$. The fully discretization reads: for $k = 1,\cdots, K-1$ and $n = 1, \cdots, N-1$, find $\{ v_k^n \} \subset V_{\text{POD}}^{\ell}$ such that 
\begin{eqnarray}
 \left ( \frac{ v_k^{n+1} - 2v_k^n + v_k^{n-1}}{(\Delta t)^2}, w; z_k \right )_c +  \left ( \frac{v_k^{n+1} - v_k^{n-1}}{2\tau}, w \right ) + \frac{1}{2} a(v_k^n,w; z_k) = \left ( \frac{v_{k-1}^{n+1} - v_{k-1}^{n-1}}{2\tau} , w \right)
\label{eqn:fully-dis}
\end{eqnarray}
for all $w \in V_{\text{POD}}^{\ell}$. In short, we discretize implicitly in $z$ and explicitly in time. 
We show that the fully discretization is stable in time. To this aim, we first recall the inverse inequality for the multiscale space, which is proved in \cite{chetverushkin2020computational}. 
\begin{proposition}[Lemma 4.5 in \cite{chetverushkin2020computational}] Assume that $\{ \chi_{j,k}^{\text{ms}} \}_{j=1}^{N_c}$ is a set of bilinear partition of unity. For any $v \in V_{\text{CEM}}$, there is a constant $C_{\text{inv}}>0$ such that
\begin{eqnarray*} \label{eqn:inv-ineq}
\norm{\nabla v}_a \leq C_{\text{inv}} H^{-1} \norm{v}. 
\end{eqnarray*}
\end{proposition}
Then, we have the following stability result for \eqref{eqn:fully-dis}. %
\begin{lemma}\label{lem:CFL}
Suppose that the following CFL condition 
$$\frac{1}{c_{\max}} - \frac{1}{4} C_{\text{inv}}^2 H^{-2} (\Delta t)^2
\geq \delta,$$
holds for some $\delta >0$. Then, we have the following stability estimate: 
$$ \frac{1}{4 \tau} \sum_{n=1}^N \norm{v_K^{n+1} - v_K^{n-1}}^2 + \sum_{k=1}^K \mathcal{E}_{N,k} \leq \frac{1}{4 \tau} \sum_{n=1}^N \norm{v_0^{n+1} - v_0^{n-1}}^2 + \sum_{k=1}^K \mathcal{E}_{0,k}.$$
\end{lemma}
%\begin{lemma}
%Under the CFL condition 
%for some $\delta >0$, we have 
%$$\mathcal{E}_{n, k} \geq  \frac{1}{4} \left (  \norm{v_k^{n+1}}_a^2 + \norm{v_k^{n}}_a^2 \right)+\delta^2 \norm{\frac{v_k^{n+1} - v_k^{n}}{\Delta t}}^2.$$%+  \frac{1}{2\tau} \norm{ v_k^{n+1} - v_k^{n-1}}^2  .$$
%\end{lemma}
\begin{proof}
Define the energy as follows: 
\begin{eqnarray*}
 \mathcal{E}_{n,k} :=   \norm{\frac{v_k^{n+1} - v_k^{n}}{\Delta t}}_c^2 + \frac{1}{2} a(v_k^{n+1} , v_k^{n}). %+\frac{1}{2\tau} \norm{ v_k^{n+1} - v_k^{n-1}}^2.
%\label{energy}
\end{eqnarray*}
Using the inverse inequality for the multiscale functions, one can show that 
\begin{align*}
\mathcal{E}_{n,k} & =  \norm{\frac{v_k^{n+1} - v_k^{n}}{\Delta t}}_c^2 + \frac{1}{2} a(v_k^{n+1} , v_k^{n})  \\
& =  \norm{\frac{v_k^{n+1} - v_k^{n}}{\Delta t}}_c^2 +  \frac{1}{4}  a(v_k^{n+1}, v_k^{n+1}) + \frac{1}{4} a(v_k^{n}, v_k^{n}) - \frac{1}{4} a(v_k^{n+1} - v_k^{n}, v_k^{n+1} - v_k^{n})\\%+\frac{1}{2\tau} \norm{ v_k^{n+1} - v_k^{n-1}}^2 \\
& \geq \norm{\frac{v_k^{n+1} - v_k^{n}}{\Delta t}}_c^2 +   
\frac{1}{4} \left ( \norm{v_k^{n+1}}_a^2 + \norm{v_k^{n}}_a^2 \right ) - \frac{1}{4}  C_{\text{inv}}^2 H^{-2} (\Delta t)^2 \norm{\frac{v_k^{n+1} - v_k^{n}}{\Delta t}}^2\\%+\frac{1}{2\tau} \norm{ v_k^{n+1} - v_k^{n-1}}^2 \\
&\geq
\big(\frac{1}{c_{\max}} - \frac{1}{4} C_{\text{inv}}^2 H^{-2} (\Delta t)^2\big) \norm{\frac{v_k^{n+1} - v_k^{n}}{\Delta t}}^2 + \frac{1}{4} \left ( \norm{v_k^{n+1}}_a^2 + \norm{v_k^{n}}_a^2 \right ) \geq 0.%+\frac{1}{2\tau} \norm{ v_k^{n+1} - v_k^{n-1}}^2 .
\end{align*}
%The lemma follows directly from the assumption. 
%Here, we have made use of the inverse inequality for the multiscale functions. 
%Recall the discrete energy 
%$$\mathcal{E}_{n,k}^2= \norm{\frac{v_k^{n+1} - v_k^{n}}{\Delta t}}_c^2 + \frac{1}{2} a(v_k^{n+1} , v_k^{n}).$$
Next, taking $w = v_k^{n+1} - v_k^{n-1}$ in \eqref{eqn:fully-dis} and denoting $f^n_{k-1} = v_{k-1}^{n+1}-v_{k-1}^{n-1}$, we have 
\begin{eqnarray*}
\frac{1}{2\tau} \norm{v_k^{n+1} - v_k^{n-1}}^2 + \norm{\frac{v_k^{n+1} - v_k^n}{\Delta t}}_c^2  - \norm{\frac{v_k^n - v_k^{n-1}}{\Delta t}}_c^2
+\frac{1}{2}a ( v_k^n , v_k^{n+1} - v_k^{n-1}) = \frac{1}{2\tau} \left (f^n_{k-1}, v_k^{n+1} - v_k^{n-1}\right).
\end{eqnarray*}
Then, it implies that 
\begin{eqnarray*}
\begin{split}
\frac{1}{2\tau} \norm{v_k^{n+1} - v_k^{n-1}}^2+ \mathcal{E}_{n,k} - \mathcal{E}_{n-1,k} & =   \frac{1}{2\tau} \left (f^n_k, v_k^{n+1} - v_k^{n-1}\right)\\
& \leq \frac{1}{2\tau}\norm{v_{k-1}^{n+1} - v_{k-1}^{n-1}} \norm{v_k^{n+1} - v_k^{n-1}} \\
& \leq \frac{1}{4\tau} \left ( \norm{v_{k-1}^{n+1} - v_{k-1}^{n-1}}^2 + \norm{v_k^{n+1} - v_k^{n-1}}^2 \right ) \\
%& \leq \frac{1}{2\delta \Delta z}\norm{f_{k-1}^n} (\mathcal{E}_{n,k} + \mathcal{E}_{n-1,k}), \\
%\implies \mathcal{E}_{n,k} - \mathcal{E}_{n-1,k} & \leq \frac{1}{2\delta \Delta z}\norm{f_{k-1}^n}. %\leq \frac{\Delta t}{2 \delta^2 \Delta z} (\mathcal{E}_{n,k-1} + \mathcal{E}_{n-1,k-1}) \\
%\implies \mathcal{E}_{n,k} - \frac{1}{2\delta^2 \Delta z} \mathcal{E}_{n,k-1} & \leq \mathcal{E}_{n-1,k} + \frac{1}{2\delta^2 \Delta z} \mathcal{E}_{n-1,k-1}.
\implies \frac{1}{4\tau}  \norm{v_k^{n+1} - v_k^{n-1}}^2+ \mathcal{E}_{n,k}  & \leq  \frac{1}{4\tau}  \norm{v_{k-1}^{n+1} - v_{k-1}^{n-1}}^2 + \mathcal{E}_{n-1,k}.
\end{split}
\end{eqnarray*}
Summing over $n = 1,\cdots, N$, we have 
$$\frac{1}{4\tau} \sum_{n=1}^N \norm{v_k^{n+1} - v_k^{n-1}}^2+ \mathcal{E}_{N,k} \leq  \frac{1}{4\tau}  \sum_{n=1}^N \norm{v_{k-1}^{n+1} - v_{k-1}^{n-1}}^2 + \mathcal{E}_{0,k}.$$
Therefore, summing over $k = 1,\cdots, K$, we have 
$$\frac{1}{4\tau} \sum_{n=1}^N \left ( \sum_{k=1}^K \norm{v_k^{n+1} - v_k^{n-1}}^2 - \sum_{k=0}^{K-1} \norm{v_{k}^{n+1} - v_{k}^{n-1}}^2 \right )+ \sum_{k=1}^K \mathcal{E}_{N,k} \leq  \sum_{k=1}^K \mathcal{E}_{0,k}$$
and 
$$ \frac{1}{4 \tau} \sum_{n=1}^N \norm{v_K^{n+1} - v_K^{n-1}}^2 + \sum_{k=1}^K \mathcal{E}_{N,k} \leq \frac{1}{4 \tau} \sum_{n=1}^N \norm{v_0^{n+1} - v_0^{n-1}}^2 + \sum_{k=1}^K \mathcal{E}_{0,k}.$$
This completes the proof. 
\end{proof}
\begin{remark}
It should be noted that the CFL condition depends on the quantity $c_{\max} := \norm{c}_{L^\infty(D \times \Omega)}$.
It also relies on the constant $C_{\text{inv}}$ which comes from the inverse inequality (see Lemma 4.5 in \cite{chetverushkin2020computational}). From the proof of that lemma, the constant $C_{\text{inv}}$ is independent of the ratio of the contrast values in the permeability. %, which is a relief of the CFL.} 
We are using the CEM-GMsFEM, which is a coarse mesh method, hence the coarse mesh size $H$ will make the CFL much less restrictive. The idea to improve the CFL condition is to use the implicit scheme in time; more precisely, one can use 
$a(v_k^{n+1},w; z_k)$ in \eqref{eqn:fully-dis}. %However, the stability is hard to be proved in this case. 
This will become one of our future works.
\end{remark}

\section{Convergence Analysis} \label{sec:analysis}
In this section, we present the convergence analysis of the proposed POD-based multiscale method for the parabolic wave equation. 
Let $v_{\ms} \in V_{\ms}$ be the semi-discretized solution which solves the problem \eqref{eqn:quasi-dis-cem}. 
%We define$$V:= \bigcup_{j = 1}^{K} V_{\ms}^j, $$where $V_{\ms}^j$ is spanned by the CEM basis functions corresponding to the node $z = z_j$.
%Recall that the bilinear form $a(\cdot,\cdot): V\times V\rightarrow \mathbb{R} $ is defined as follows: $$a(w, u)= (w, u)_{a} = \int_{\Omega}c\nabla_x w\cdot\nabla_x u ~ dx,$$ for all $w, u\in V$. % and $k = 1,..., K+1$. 
Throughout this section, we assume that the coercivity and boundedness of this bilinear form hold. Specifically, there exist two constants $\gamma$ and $\beta$, independent of any $z_k \in D$,  such that
\begin{align}
\gamma \|u\|_{H_1(\Omega)}^2\leq a(u, u; z_k) \tand
a(u, w; z_k) \leq \beta \|u\|_{H_1(\Omega)}\|w\|_{H_1(\Omega)}
\end{align}
for any $u, w\in V$. We denote $(\cdot,\cdot)_{H^1}$ the $H^1$-inner product such that $(u,v)_{H^1} := (\nabla u, \nabla v) + (u,v)$ for any $u, v \in V$ with its associated norm being $\norm{\cdot}_{H^1} := \sqrt{(\cdot,\cdot)_{H^1}}$. 

%Next, we define the snapshot space $\mathcal{V}$ by including: $$y_j := v_{\ms} (t, z_{j-1}), \quad  y_{j+K+1} := \partial_{tt} v_{\ms}(t, z_{j-1}), \quad j = 1,\cdots, K+1, $$ and $$y_{j+2K+2} := \Tilde{\partial} \dot{v}_{\ms}(t, z_j) = \frac{\dot{v}_{\ms}(t, z_j)-\dot{v}_{\ms}(t, z_{j-1})}{\Delta z}, \quad j = 1,\cdots ,K.$$ Here, we have $v_{\ms} (t, z_j), \dot{v}_{\ms}(t, z_j) := \partial_t v_{\ms} (t, z_j)$, and $\partial_{tt} v_{\ms}(t, z_j)\in  V$ for all $t\in [0, T]$.Then, we perform the POD procedure on the snapshot space $\mathcal{V}$ and we denote $\{\psi_k\}_{k = 1}^\mathcal{N}$ the corresponding POD basis functions. We then define $V_{\text{POD}}^\ell := \spa\{\psi_k\}_{k = 1}^\ell$ for a given positive number $\ell \leq \mathcal{N}$. 
We have the following lemma for the proper orthogonal decomposition. %with the Hilbert space $X = V_{\text{CEM}}$ and $\norm{v}_X := \norm{v}_{L^2(0,T;H^1(\Omega))}$. 

\begin{lemma}[cf. Proposition 1 %and equation (7) 
in \cite{kunisch2001galerkin}]
Let $v_{\ms} \in V_{\ms}$ be the semi-discretized solution which solves the problem \eqref{eqn:quasi-dis-cem} and  
$K_{ij} := \frac{1}{3K+2}(y_j, y_i)_{H^1}$ be the correlation matrix and $\lambda_k$ be the corresponding eigenvalues sorted ascending. For any integer $\ell$ with $0< \ell \leq \mathcal{N}$, the following error formula holds: 
\begin{eqnarray}
\begin{split}
   \frac{1}{3K+2} & \left [ \int_0^T \sum_{k = 1}^{K+1}  \norm{v_{\ms} (t, z_k)-\sum_{j = 1}^{\ell} \left (v_{\ms}(t, z_k), \psi_j\right )_{H^1}\psi_j}_{H^1}^2 \right . \\
& + \left . \sum_{k = 1}^K \norm{\Tilde{\partial} \dot{v}_{\ms}(t, z_k)-\sum_{j = 1}^{\ell} \left (\Tilde{\partial}\dot{v}_{\ms}(t, z_k), \psi_j\right)_{H^1}\psi_j }_{H^1}^2 \right . \\
& \left . + \sum_{k = 1}^{K+1} \norm{\ddot{v}_{\ms}(t, z_k)-\sum_{j = 1}^{\ell} \left (\ddot{v}_{\ms}(t, z_k), \psi_j\right)_{H^1}\psi_j}_{H^1}^2 dt \right ] = \sum_{k = \ell+1}^
\mathcal{N} \lambda_k. 
\end{split}
\end{eqnarray}
\label{pod}
\end{lemma}

Next, we define the Ritz-projection $P^\ell: V \rightarrow V_{\text{POD}}^\ell$ by: 
\begin{align}
    (P^\ell v, \psi)_{H^1} = (v, \psi)_{H^1} \quad \tforall  \psi\in V_{\text{POD}}^\ell,
    \label{eqn:pod-proj}
\end{align}
for any $v\in V$. We have the following estimate for the projection operator. 
\begin{comment}
Since the bilinear form is bounded and coercive, it follows that 
\begin{align*}
     \norm{P^\ell v}_{a}^2 &\leq 
     \beta \norm{P^\ell v}_{H_1}^2 \leq
     \beta \gamma^{-1} a\left (P^\ell v, P^\ell v\right) \\
     & = \beta \gamma^{-1} a\left (v,P^\ell v \right) \leq \beta \gamma^{-1} \norm{v}_a \norm{P^\ell v}_a
\end{align*}
where we have used the property \eqref{eqn:pod-proj}. 
%    &\leq
%    \frac{\Gamma^2\beta}{\kappa} \|P^l v\|_{H_1}^2 
%    \leq 
%    \frac{\Gamma^2\beta}{\kappa\gamma^2}\|P^l v \|_{a}^2
%    \leq \frac{\Gamma^2\beta}{\kappa\gamma^2} \|v\|_{a}^2 ,
Therefore, we have 
$\norm{P^\ell v}_{a}\lesssim \norm{v}_{a}$ for any $v \in V$. %hence the projection is well defined.
\end{comment}

\begin{lemma}
Let $\{z_k\}_{k = 1}^{K+1}\subset D$. 
For any $\ell \in\{1,2,...,\mathcal{N}\}$ and $v \in V$, the projection operator $P^\ell$ satisfies the following estimate: 
\begin{gather}
    \frac{1}{K+1} \sum_{k = 1}^{K+1} \int_0^T \norm{v_{\ms}(z_k)-P^\ell v_{\ms}(z_k)}_{H^1}^2 ~ dt \lesssim \sum_{k = \ell+1}^{\mathcal{N}} \lambda_k.
\end{gather}
\label{lemma1}
\end{lemma}

\begin{proof}
Let $v\in V$ be arbitrary. By the definition of the projection $P^\ell$, we have
\begin{eqnarray*}
\begin{split}
    %\kappa 
    \norm{v - P^\ell v}_{H^1}^2  = (v - P^\ell v, v - P^\ell v)_{H^1}
     = (v- P^l v, v-\psi)_{H^1} 
%    & \leq \beta\Gamma^2 \|v-P^l v\|_{H_1} \|v-\psi\|_{H_1} \\ 
\leq \norm{v-P^\ell v}_{H^1} \norm{v-\psi}_{H^1}
\end{split}
\end{eqnarray*}
for all $\psi\in V_{\text{POD}}^\ell$. 
This implies that,
\begin{align}
\norm{v-P^\ell v}_{H^1}
    \lesssim \norm{v-\psi}_{H^1} \quad \tforall \psi \in V_{\text{POD}}^\ell.
    \label{eqn2}
\end{align} 
The lemma follows immediately by Lemma \ref{pod} and \eqref{eqn2}. %, we have
\begin{comment}
\begin{align*}
    \frac{1}{K+1}\sum_{k = 1}^{K+1} \norm{v_{\ms}(z_k)-P^\ell v_{\ms}(z_k)}_a^2 
    &\lesssim
    \frac{1}{K+1}
    \sum_{k = 1}^{K+1} \norm{v_{\ms}(z_k)-\sum_{i = 1}^{\ell}
    \left (v_{\ms} (z_k), \psi_i \right )_a\psi_i}_a^2\\
    &\lesssim
    \frac{4}{(3K+2)}
    \sum_{k = 1}^{K+1} \norm{v(z_k)-\sum_{i = 1}^{\ell}
    (v_{\ms}(z_k), \psi_i)_a\psi_i}_a^2\\
    &\lesssim
    \sum_{i = l+1}^d \lambda_i.
\end{align*}
\end{comment}
\end{proof}

\begin{comment}
\begin{remark}
If $v(t, z, x)$ solves the semi-discretized problem, that is, $v(t, z_k)\in V$ for given $t\in [0, T]$, we then have,
\begin{align}
    \frac{1}{K+1}\sum_{k = 1}^{K+1} \int_0^T \|v(z_k)-P^l v(z_k)\|_{a}^2dt \lesssim \sum_{k = l+1}^d \lambda_k.
\end{align}
\end{remark}
\end{comment}

Clearly, from the proof of Lemma \ref{lemma1}, the corollary below follows immediately. 
\begin{corollary}
Let $v_{\ms} \in V_{\ms}$ be the semi-discretized solution which solves \eqref{eqn:quasi-dis-cem}. 

Then, the following estimates hold: 
\begin{align}
    \frac{1}{K}\sum_{k = 1}^K \int_0^T  \norm{\Tilde{\partial} \dot{v}_{\ms}(z_k)-P^\ell \Tilde{\partial} \dot{v}_{\ms}(z_k)}_{H^1}^2 ~ dt \lesssim \sum_{i = \ell+1}^{\mathcal{N}} \lambda_i
\label{pod_eqn}
\end{align}
and 
\begin{align}
    \frac{1}{K+1}\sum_{k = 0}^{K} \int_0^T \norm{\ddot{v}_{\ms}(z_k)-P^\ell \ddot{v}_{\ms}(z_k)}_{H^1}^2 dt \lesssim \sum_{i = \ell+1}^{\mathcal{N}} \lambda_i.
\end{align}
\label{pod_cor}
\end{corollary}
Let $\{U_k\}_{k = 0}^K \subset V_{\text{POD}}^{\ell}$ be the semi-discretized solution satisfying the following equation 
\begin{align}
( \ddot{U}_k, \phi; z_k)_c + (\tilde \partial \dot{U}_k, \phi)+\frac{1}{2} a(U_k, \phi; z_k) = 0
\label{semi_dis_pod}
\end{align}
for all $\phi\in V_{\text{POD}}^\ell$ with appropriate initial condition $U_0$. Here, we denote $\tilde \partial \dot{U}_k := \frac{\dot{U}_k - \dot{U}_{k-1}}{\Delta z}$. 
We are now able to analyze the error. Assume that $\{ v_k^n \} \subset V_{\text{POD}}^\ell$ solves \eqref{eqn:fully-dis}, $\{U_k\}_{k=0}^K \subset V_{\text{POD}}^\ell$ solves \eqref{semi_dis_pod}, and the function $v_{\ms}\in V_{\ms}$ solves \eqref{eqn:quasi-dis-cem}. 
We decompose the error into three parts: 
$$ v_k^n - v_{\ms}(t_n) = \underbrace{v_k^n - U_k(t_n)}_{=: \mu_k^n} + \underbrace{U_k(t_n) - P^\ell v_{\ms}(t_n)}_{=: \nu_k^n} + \underbrace{P^\ell v_{\ms}(t_n) - v_{\ms}(t_n)}_{ =: \rho_k^n}.$$

%$$U_k-v(z_k) = U_k-P^l v(z_k )+P^l v(z_k)-v(z_k) = \nu_k+\rho_k.$$
We denote $\mu_k(t)$, $\nu_k(t)$, and $\rho_k(t)$ the piecewise linear functions that interpolates $\{ \mu_k^n \}$, $\{ \nu_k^n \}$, and $\{ \rho_k^n \}$ in time, respectively. Due to Lemma \ref{lemma1}, we have
\begin{align*}
    \frac{1}{K}\sum_{k = 1}^K \norm{\rho_k}_{L^2(0,T;H^1(\Omega))}^2 \lesssim \sum_{k = \ell+1}^{\mathcal{N}} \lambda_k.
\end{align*}
It should be noted that the term $\sum_{k = \ell+1}^
\mathcal{N} \lambda_k$ comes from the method of POD and is a typical error term in the POD analysis. The error decay to $0$ very fast; and by the theory, this error is optimal \cite{kunisch2001galerkin} since the method of POD solves a minimization problem \eqref{eqn:pod-exp}.

Next, using the notation $\Tilde{\partial}\nu_k = \frac{\nu_k-\nu_{k-1}}{\Delta z}$ for all $k = 1,..., K$ and the equation \eqref{eqn:quasi-dis-cem}, we obtain
\begin{align*}
    (\ddot{\nu}_k, \psi;z_k)_c+(\Tilde{\partial}\dot{\nu}_k, \psi )+
    \frac{1}{2} a(\nu_k, \psi; z_k) &=
    -(P^\ell \ddot{v}_{\ms}, \psi; z_k)_c-(\Tilde{\partial}P^\ell  \dot{v}_{\ms},\psi)
    -\frac{1}{2} a(v_{\ms}, \psi; z_k)\\
    & = (\ddot{v}_{\ms}-P^\ell \ddot{v}_{\ms}, \psi; z_k)_c + (\tilde \partial \dot{v}_{\ms}-\tilde{\partial}P^\ell \dot{v}_{\ms}, \psi)
\end{align*}
for any $\psi \in V_{\text{POD}}^\ell$. Let us denote
$$
h_k := \tilde \partial \dot{v}_{\ms}-\tilde{\partial}P^\ell \dot{v}_{\ms}
\quad 
\text{and} \quad 
w_k := \ddot{v}_{\ms}-P^\ell \ddot{v}_{\ms}.$$ 
By Corollary \ref{pod_cor}, it follows that
\begin{align*}
%    \sum_{k = 1}^K \int_0^T     \norm{h_k}^2 dt + \sum_{k = 0}^K \int_0^T \|w_k\|^2 dt    \leq 
    \frac{1}{K}\sum_{k = 1}^K 
    \norm{h_k}_{L^2(0,T;H^1(\Omega))}^2 + \frac{1}{K+1} \sum_{k = 0}^K \norm{w_k}_{L^2(0,T;H^1(\Omega))}^2
    \lesssim \sum_{i = \ell+1}^{\mathcal{N}} \lambda_i.
\end{align*}
%By using Taylor's expansion,
%\begin{align*}
%\sum_{k = 1}^K \norm{g_k}_{L^2(0,T;H^1(\Omega))}^2 = \sum_{k = 1}^K\int_0^T \mathcal{O}(\Delta z^2) dt \lesssim T\mathcal{O}(\Delta z).
%\end{align*}
Using the same technique of showing Lemma \ref{semi_stability}, one can obtain an estimate for $\nu_k$: 
\begin{align*}
\|\dot{\nu}_k\|_{L^2(0,T;L^2(\Omega))}^2+\Delta z\sum_{k = 1}^K  \mathcal{E}_k(\nu_k(T))\lesssim
\Delta z\sum_{k = 1}^K
\norm{h_k}_{L^2(0,T;H^1(\Omega))}^2 + \norm{w_k}_{L^2(0,T;H^1(\Omega))}^2 \lesssim  \sum_{i=\ell+1}^{\mathcal{N}} \lambda_i.
\end{align*}
It remains to estimate the term $\mu_k$. This error comes from the temporal discretization and it follows that 
$$ \norm{\mu_k}_{L^2(0,T;H^1(\Omega))}^2 = \int_0^T \norm{\mu_k}_{H^1}^2 dt = \int_0^T O(\tau^2)dt = O(\tau).$$
Denote $v_k(t)$ the piecewise linear function that interpolates $\{ v_k^n \}$ in time. As a result, we have the following error estimate 
\begin{align}
\frac{1}{K}\sum_{k=1}^K \norm{v_k - v_{\ms}}_{L^2(0,T;H^1(\Omega))}^2 \lesssim O(\tau) + \sum_{i=\ell+1}^{\mathcal{N}} \lambda_i.
\end{align}

We remark that the error between the solution $\theta$ (that solves \eqref{eqn:quasi-dis}) and the solution $v_{\ms}$ has the relation: $\norm{\theta - v_{\ms}}_{L^2(0,T;H^1(\Omega))} = O(H)$. Therefore, the whole error estimate reads as follows: 
\begin{align}
\frac{1}{K} \sum_{k=1}^K \norm{\theta - v_k}_{L^2(0,T;H^1(\Omega))}^2 \lesssim O(H^2) + O(\tau) + \sum_{i=\ell+1}^{\mathcal{N}} \lambda_i.
\end{align}

\section{Numerical Experiments} \label{sec:numerics}
In this section, we present some numerical results to demonstrate the efficiency of the proposed computational multiscale method. 
In the experiments below, we set the spatial domain to be $\Omega = (0,1)^2$. 
We partition the spatial domain into $10 \times 10$ uniform square elements and we refer this partition to the coarse grid with mesh size $H = \sqrt{2}/10$. Further, we divide each coarse element into $10 \times 10$ uniform square elements and the corresponding fine grid has resolution of the size $100 \times 100$. We refer this to the fine grid of the spatial domain with mesh size $h  =\sqrt{2}/100$. 
%In the dimension related to the spatial domain $\Omega$, we discretize the domain using the POD-based CEM-GMsFEM. 
%The fine mesh size is $h = \frac{1}{100}$ and the coarse mesh is $H = 10h$. 
We compute the numerical solutions $v_k^N$ at the terminal time for all $k = 1, \cdots, K$ and compare it with the reference solutions $v_f(z_k,x,T)$ for all $k$, which are computed by using the underlying fine grid. 
We measure the relative error in $L^2$ norm at each layer $z_k$, which is defined as follows:
$$
 e_2^k := \frac{\norm{v_k^N- v_f(z_k,x,T)}}{\norm{v_f(z_k,x,T)}}. 
$$

\subsection{The First Experiment}
%In the experiment below, we set the terminal time $T = 0.05$ and $\Omega = (0, 1)^2$. 
In the first experiment, we set the initial condition to be $v(0,x,t) = \sin(\pi x_1)\sin(\pi x_2)\sin(t)$ where $x = (x_1, x_2)\in \Omega$; also we set homogeneous initial condition for all $z$, i.e., $v(z, x, 0) = 0$ for all $(z, x)\in D \times \Omega$. 
The temporal direction and $z$-direction are discretized as in the scheme \eqref{eqn:fully-dis}. We choose the time step to be $\Delta t =10^{-5}$ and the step size in $z$-direction is $\Delta z = 10^{-4}$. 
We set $K = 30$, i.e., the $z$ direction is partitioned into $30$ levels and $D = [0, 30 \Delta z]$.
 
In this experiment, the pattern of the permeability field $c(x)$ is given in Figure \ref{kappa} (Left). 
In $x_1$- and $x_2$-directions, the permeability will follow the same pattern. However, along the $z$-direction, the permeability field $c(x)$ has different values of contrast. See Figure \ref{kappa} (Right) for an illustration for this layer-structured heterogeneous media. In particular, we partition the permeability field along the $z$-direction into $3$ different layers; the value of contrast in the yellow region within each layer are equal to $10$, $15$, and $20$, respectively. We remark that the patterns keep the same as demonstrated in Figure \ref{kappa} (Left).

To obtain the set of POD basis functions, we first construct the CEM basis functions for each layer and then obtain the reduced basis functions by performing the POD procedure on the combined basis space. We choose 100 POD basis functions from the total $300$ CEM basis functions in this experiment. We remark that the number of the POD reduced basis functions depends on the decay of the eigenvalues $\{ \lambda_k \}$ corresponding to the POD construction. In practice, one has to include sufficiently many reduced basis functions so that the tail of the eigenvalues is smaller than a given tolerance of accuracy. 

\begin{figure}[ht]
\centering
\mbox{
\includegraphics[scale = 0.4]{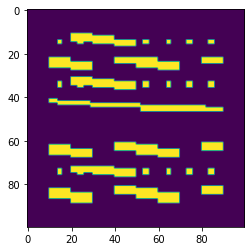} \qquad 
\includegraphics[scale = 0.2]{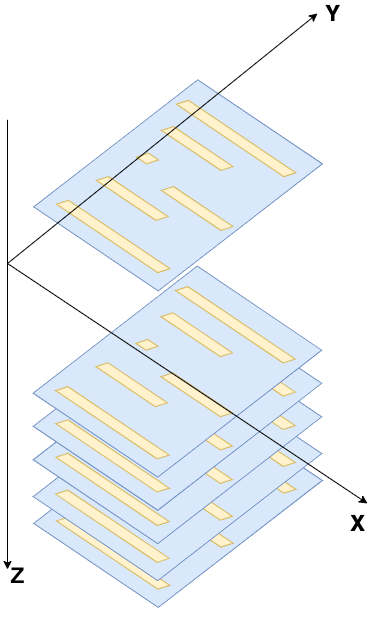}}
\caption{Left: The two-dimensional pattern of permeability field in Experiment $1$. %$c$ in $\Omega$. 
Right: Permeability field at different layers $z_k$ has the same pattern but different values of contrast.}
\label{kappa}
\end{figure}

The record of relative error $e_2^k$ for all $k \in \{ 1, \cdots, K\}$ at the terminal time is plotted in Figure \ref{loss1} (Left). 
The relative error at each layer is around the level of magnitude of $10^{-3}$. 
One can observe from the graph that the proposed numerical scheme is capable for accurately approximating the fine-scale solution. 

\begin{figure}[ht]
\centering
\mbox{
\includegraphics[scale = 0.15]{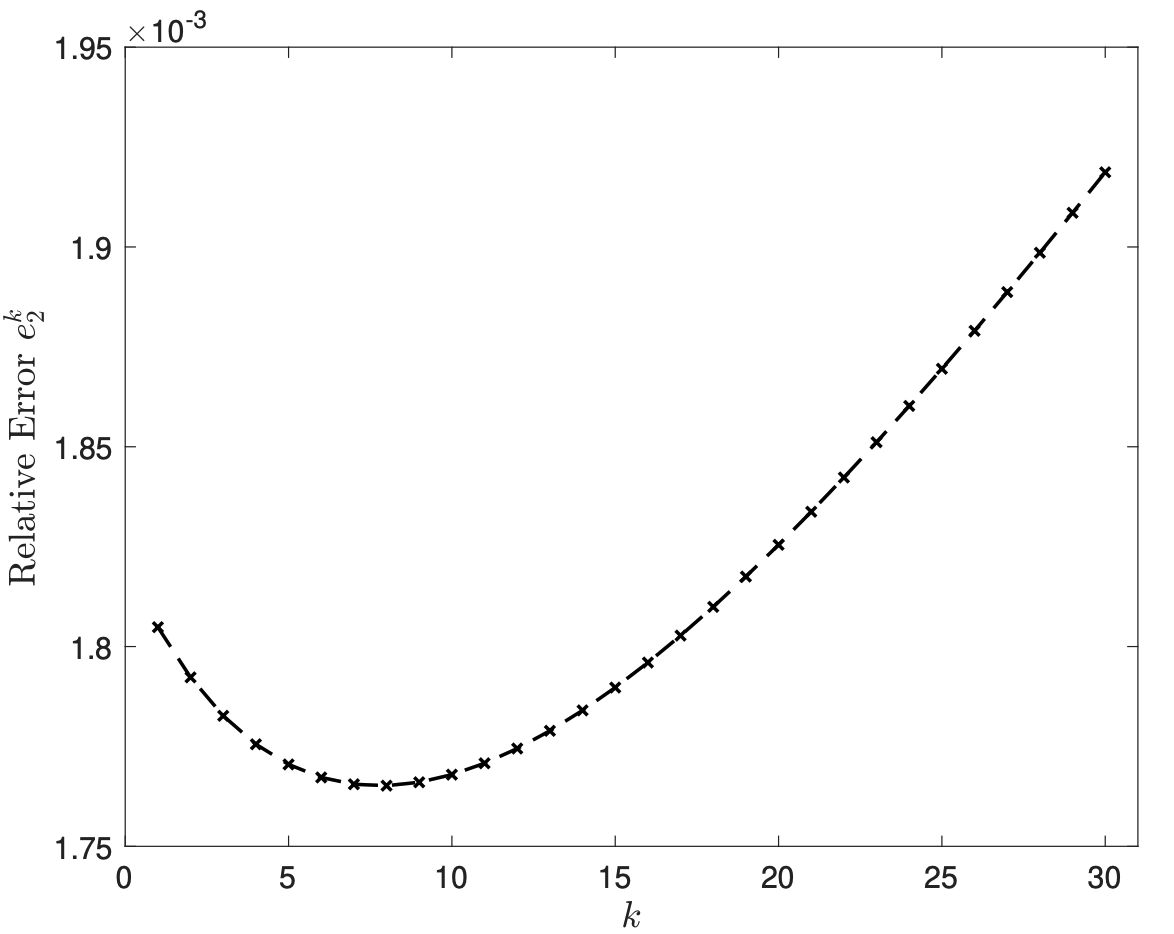} \qquad 
\includegraphics[scale = 0.15]{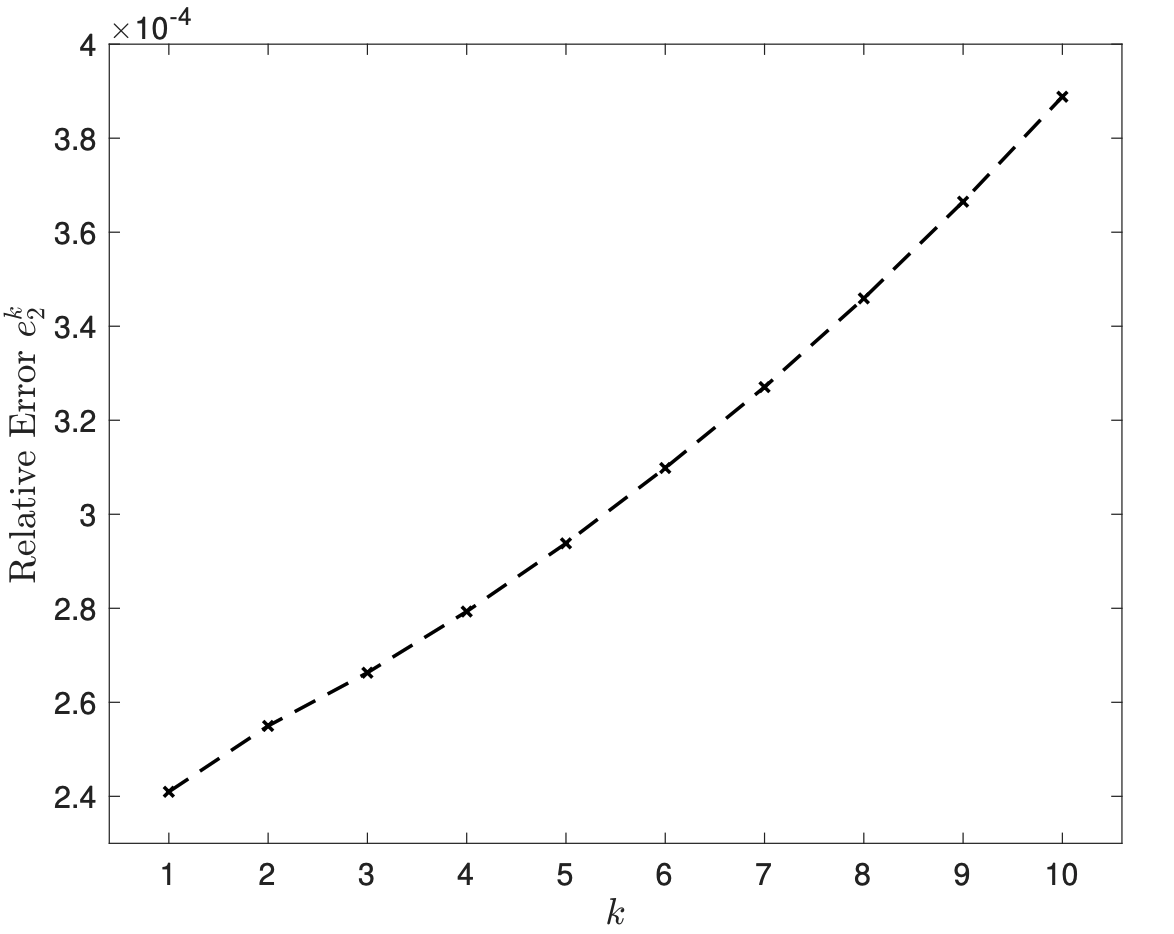}}
\caption{The relative error $e_2^k$. Left: Experiment 1; Right: Experiment 2.}% $Y$ axis is the relative error. $X$ axis is the $z$ direction. We measure the relative error at the terminal time of each $z$ grid points}
\label{loss1}
\end{figure}

\subsection{Experiment 2}
In the second experiment, we consider the so-called Marmousi heterogeneous field.
More precisely, given $z_k\in D$, $c(z_k)$ is a Marmousi permeability field as shown in Figure \ref{Marmousi}.
\begin{figure}[ht]
\centering
\includegraphics[scale = 0.4]{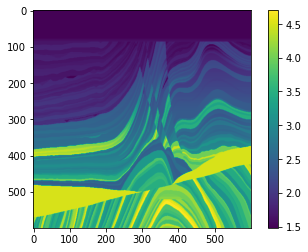}
\caption{Marmousi permeability field in Experiment $2$.}% Each $c(z_k)$ is obtained by taking a fine element in each coarse element of size $6\times 6$ in this fine grid field of size $600 \times 600$}
\label{Marmousi}
\end{figure}

%In this example,  z
The initial and boundary data are the same as in the first experiment. 
We set $T = 5 \times 10^{-3} $ and $\Delta t = 5\times 10^{-6}$. We set $\Delta z = 10^{-4}$ and $D = [0, 10 \Delta z]$ with $K = 10$. 
For each $z_k \in D$, the medium $c(z_k)$ is obtained by taking the data from a very fine Marmousi field of size $600\times 600$.
To be more specific, we partition the fine field (which is of size $600\times 600$) into a coarse field of size $100\times 100$ whose coarse element size is $6\times 6$. Then, in each coarse element, we randomly pick a value in one of the fine elements in this coarse element; and we formulate our computational fine grid ($100\times 100$) as combining all selected elements accordingly.

To get the set of POD basis functions, we first calculate three sets of CEM basis functions  with $k \in \{ 0, 4, 8\}$ and obtain the POD basis by performing the POD procedure on the combined set of CEM basis functions. We choose $50$ POD basis functions from the total $300$ CEM basis functions. The record of relative error for all $z_k$ at the terminal time is plotted in Figure \ref{loss1} (Right). 
The relative error at each layer is around the level of magnitude of $10^{-4}$ in this case. In particular, the relative error $e_2^k$ is about $3.9 \times 10^{-4}$ for $k = 10$. This demonstrates the efficiency of the proposed multiscale method for the simulation of parabolic wave formulation. 

\begin{comment}
\begin{figure}[ht]
\centering
\includegraphics[scale = 0.4]{graphs/loss2.eps}
\caption{The relative error at terminal time against $z$-direction in Experiment $2$}%. Here, the $Y$ axis is the relative error and the $X$ axis is the $z$-direction. We measure the relative error at the terminal time of each $z$ grid points}
\label{loss2}
\end{figure}
\end{comment}

\section{Conclusion} \label{sec:conclusion}
In this paper, we developed a computational multiscale method for simulating the parabolic wave formulation with highly heterogeneous media. For the spatial discretization, we employed the recently developed CEM-GMsFEM, which is proved to be efficient to reduce the dimension of the model in space. We then combined the technique of proper orthogonal decomposition to further reduce the dimension along the quasi time direction. A complete analysis of the proposed algorithm has been provided. Numerical results are provided to demonstrate the effectiveness and efficiency of the proposed multiscale method. 

\section*{Acknowledgement}

\bibliographystyle{abbrv}
\bibliography{references}%, references1}%, references2, references3, references4}
\end{document}